\documentclass[12pt,leqno,twoside]{amsart}
\usepackage{amsmath}
\usepackage{amssymb}
\usepackage{mathrsfs}
\usepackage[english]{babel}
\usepackage{ot1patch}

\usepackage{amstext,amsmath,amscd, bezier,indentfirst,amsthm,amsgen,enumerate, geometry}
\usepackage[all,knot,arc,import,poly]{xy}
\usepackage{amsfonts,color}

\usepackage{srcltx}
\usepackage{graphicx}
\usepackage{epsfig}

\newtheorem{theorem}{Theorem}[section]
\newtheorem{lemma}[theorem]{Lemma}
\newtheorem{corollary}[theorem]{Corollary}
\newtheorem{proposition}[theorem]{Proposition}

\theoremstyle{remark}
\newtheorem{remark}[theorem]{\sc Remark}
\newtheorem*{remark*}{Remark}

\theoremstyle{definition}
\newtheorem{definition}[theorem]{Definition}
\newtheorem{example}[theorem]{Example}

\numberwithin{equation}{section}

\newcommand{\e}{\varepsilon}

\newcommand{\Sing}{\mathrm{Sing } }

\newcommand{\bR}{{\mathbb R}}
\newcommand{\bC}{{\mathbb C}}

\newcommand{\bQ}{{\mathbb Q}}

\newcommand{\fin}{\hspace*{\fill}$\Box$}

\begin{document}

\title[Non normally embedded complex spaces]{Testing Lipschitz non normally embedded \\ complex spaces}

\author{\sc Maciej  Denkowski}

\address{Jagiellonian University, Faculty of Mathematics and Computer Science, Institute of Mathematics, \L ojasiewicza 6, 30-348 Krak\'ow, Poland}

\email{maciej.denkowski@uj.edu.pl}

\author{\sc Mihai Tib\u ar}

\address{Math\' ematiques, 
Universit\'e de Lille,  CNRS - Laboratoire Paul Painlev\' e, F-59000 Lille, France}

\email{tibar@math.univ-lille1.fr}
\thanks{The first named author acknowledges the partial support of Polish Ministry of Science and Higher Education grant 1095/MOB/2013/0. Both authors acknowledge the support of the Labex CEMPI (ANR-11-LABX-0007-01) at  the University of Lille.}


\begin{abstract}
We introduce a sectional criterion for testing if complex analytic germs $(X,0) \subset (\bC^n, 0)$ are Lipschitz non normally embedded. 
\end{abstract}
\maketitle

\section{Introduction}
A locally closed path-connected set $X\subset{\bR}^n$ may be endowed with two metrics: the Euclidean induced metric (or \emph{outer metric}), which we shall denote by $d(x,y)=||x-y||$, and the \textit{inner} metric $d_X(x,y)$ that is defined as the shortest length of arcs  joining the points  $x, y \in X$. We have $d(x,y)\leq d_X(x,y)$ for all $x,y\in X$. If $X$ is subanalytic or definable in some o-minimal structure (we refer the reader to \cite{DS}, \cite{C} for these notions), then $d_X$ is a well-defined, finite metric. 

  As introduced by  \cite{BM}, one says that $X \subset \bR^n$ is \textit{(Lipschitz) normally embedded}  if $d_X(x,y)\leq Cd(x,y)$, for some constant $C>0$, for all $x,y\in X$.  Many papers have been dedicated to this property ever since, with a raise of interest in the very last years, e.g.  \cite{BFN-normemb, B, FS1, BMe}. In the following we shall abbreviate ``normally embedded'' by NE.
 
The definition of NE amounts to saying that the outer and inner metrics are equivalent, or that the identity mapping $\mathrm{id}:  [X,d]\to [X,d_X]$ is bi-Lipschitz. 
We say that the set $X\subset{\bR}^n$ is normally embedded at $x\in X$, if for some neighborhood $U\ni x$, $U\cap X$ is normally embedded in $U$.  In this paper we focus on complex space germs $(X,0) \subset (\bC^{n}, 0)$.

 If the germ $(X,0)$ is non-singular (more generally, if it is Lipschitz regular, see e.g. \cite[Lemma 2.1]{B}) then it  is NE, but the converse is not true in general.  The simplest example is the curve germ $X: x^{2}- y^{2} =0$ which is singular at 0 but NE.
In case of  curves, it is known that  an irreducible plane or space curve is homeomorphic 
to $\bC$ via a Puiseux parametrization, but this is not necessarily bi-Lipschitz.  

We first present a short elementary proof of the following simple characterization of NE space curve germs, which will be used later: 

\begin{proposition}\label{t:planecurve}
The space curve  $(X,0)\subset ({\bC}^{n},0)$ is normally embedded if and only if all its irreducible components 
are nonsingular, meeting transversally two-by-two at the origin.
\end{proposition}
This result, which seems to be well-known to specialists,  is related to early work by Pham and Teissier \cite{PT}  about space curves and Lipschitz property in purely algebraic terms.    Our proof is analytic.  It shows in particular that, in case of space curves, the property of being NE does not depend on the embedding. 

\smallskip
\begin{remark}\label{t:embedding}
 The independence  of the NE-property of the embedding is actually a general fact, holding for any complex or real analytic space (although we could not trace back any reference for it):
 
 \emph{Let $h: (X,0) \to (Y,0)$ be a bi-holomorphism of analytic space germs, for some $(X,0)\subset (\bC^n,0)$ and $(Y,0)\subset (\bC^m,0)$. Then  $X$ is NE at 0 if and only if $Y$ is NE at 0. }
 
This is an immediate consequence of the following statement which also seems to be folklore:

\emph{Let $h\colon (X,0)\to (Y,0)$ be a bijective definable map germ which is Lipschitz with respect to the outer metrics. Then $h$ is Lipschitz with respect to the inner metrics and with the same Lipschitz constant.}
\end{remark}
\smallskip

In this note we introduce a sectional technique for testing the non-NE property.  The method of slicing by linear subspaces is 
very common in algebraic geometry, and has a long and successful history starting with the Lefschetz pencil technique. It provides valuable information about the space from one slice or from a family of slices. 
Our aim was to bring this tool into the Lipschitz geometry of singular spaces.  The main result of our paper is the following  sectional criterion for non-NE (see Definition \ref{d:admissible} for the admissibility condition):

\begin{theorem}\label{t:sectional}
Let $(X, 0)\subset (\bC^n,0)$ be a reduced singular analytic germ of pure dimension $k$, with $1<k<n$. 
If there is  an admissible $(n-k+1)$-plane  $P$  such that there is a
curve germ $\Gamma \subset P\cap X$  at $0$ which is  non-NE,
then $X$ is non-NE at $0$.
\end{theorem}

This easy-to-check criterion gives immediately many new examples of non-normally embedded varieties.  As its proof looks simple, one might think that the slicing procedure will do even without any condition. 
Our Example \ref{ex:nongen} shows this is not true:    the admissibility assumption (Definition \ref{d:admissible})  is sharp.

 Our result may be compared to the preprint \cite{BMe} by Birbrair and Mendes, uploaded on arXiv a few months after our preprint \cite{DT},  where the authors reduce the  NE property of a real space germ $(X,0)$  to the comparison of the two types of distances along any pair of arcs on the space. Proving that a space $(X,0)$ is NE remains nevertheless a task of high difficulty. Conversely, testing that $(X,0)$ is non-NE by comparing the distances along well chosen pairs of arcs is the method which has been used for finding all the examples in the literature, and it 
 is in this respect that our result contributes with an effective and sharp sectional test in the context of complex analytic space germs.  
 
  We apply our slice criterion to show that certain Brieskorn-type singular hypersurfaces are non-NE (Corollary \ref{Brieskorn}), extending a result from \cite{BFN-normemb}. 
 
\smallskip

We also derive the following observation. Let $C_0(X) = \cup_{i} C_{i}$ be the decomposition of the tangent cone  into irreducible components of dimension $k$ (see \S \ref{s:prelim})
with multiplicities $\ge 1$. 

 \begin{corollary}\label{c:admiss}  
If the tangent cone $C_0(X)$ contains some irreducible component of multiplicity $\ge 2$, then $X$ is non-NE at 0.
 \end{corollary}
 
 This  corollary is related to a recent result by L. Birbrair,  A. Fernandes, L\^e D.T. and  E. Sampaio  \cite[Proposition 3.4]{B}  saying that if the tangent cone of a singular reduced germ $(X, 0)$  is a linear space then X is non-NE at 0. New developments in this directions including a generalisation of Corollary \ref{c:admiss} have been announced in the preprints \cite{FS1, FS2} published shortly after our preprint \cite{DT}.
 
 \begin{example} For  the surface $X: x^{3} + x^{2}y + y^{3}z + z^{5}= 0$, the tangent cone is $C_0(X) : x^{2}(x+y)=0$, thus contains the 
 complex plane $\{ x=0\} \subset \bC^{3}$ with multiplicity 2.  By Corollary \ref{c:admiss}, $X$ is non-NE at 0.
 
 This surface is normal, but if we drop the term $y^{3}z$ from its equation, then the new surface has non-isolated singularity. However, its behaviour is the same, i.e. it is  non-NE at 0, by the same reason.
 \end{example}

\section{Preliminaries}\label{s:prelim}

Let $(X, 0)\subset (\bC^n,0)$ be a reduced analytic germ of pure dimension $k$, with $1<k<n$, such that $\Sing X \not= \emptyset$ as germ at 0.

 Its  {\it (Peano) tangent cone} at $0\in X$ is by definition  (see \cite[8.1, pag. 71]{Ch}):
\[
C_0(X) :=\{v\in{\bR}^n\mid \exists (x_\nu)\subset X, \lambda_\nu>0\colon x_\nu\to 0, \lambda_\nu x_\nu\to v\}.
\]
It is well-known that for complex spaces the tangent cone is a complex cone of pure dimension $k$, see e.g. \cite[8.3 Corollary, pag. 83]{Ch}.


We consider a linear surjection $\pi: \bC^{n}\to \bC^{k}$  such that the restriction $\pi_{X}: (X, 0) \to (\bC^{k}, 0)$
 is \emph{finite}, i.e. that  $X \cap  \ker \pi = \{0\}$ as germ at the origin, where $\dim \ker \pi = n-k$.
It follows that  $\pi_{X}$ is a branched covering of degree $d\ge \deg_0 X \ge 2$, since $X$ is singular at 0 by our assumption. The equality $d = \deg_0 X$ holds whenever ${\pi}$ is ``general'', in the following sense:

\begin{definition}\label{d:generalpi}
We say that  the projection $\pi$ is \emph{general} with respect to $(X,0)$ if $C_0(X)\cap  \ker \pi=\{0\}$.
\end{definition}

The \emph{ramification locus} of $\pi$ is denoted by  $\Delta_\pi\subset \bC^{k}$; it is an analytic set germ at $0\in \bC^{k}$ of dimension $< k$ and, under our assumption $\Sing X \not= \emptyset$, it is non-empty.

Let us denote by $d(y,E)$ the Euclidean distance of a point $y\in \bC^{k}$ to a set $E$. By  $B_{\e}$ we denote as usual an open ball centered at $0$ of radius $\e >0$.  We have:
\begin{lemma}\label{inner}
 For some small enough $\e >0$, let  $x_1, x_{2}\in X\cap B_{\e}$ be  two distinct points 
 such that $\pi(x_{1}) = \pi(x_{2}) = y \not\in \Delta_\pi$. Then
\begin{equation}\label{eq:dist}
d_X(x_1, x_2)\geq 2 d(y, \Delta_\pi).
\end{equation}
\end{lemma}
\begin{proof}
Let $B(y, d(y, \Delta_{\pi}))\subset \bC^{k}$ denote the open ball of radius  $d(y, \Delta_{\pi})$ centred at $y$. Let $\gamma\subset X\cap B_{\e}$ be some path joining $x_1$ to $x_2$. Its projection $\pi(\gamma)$ is a loop in $\bC^{k}$ with base point $y$.  
Then the image $\pi(\gamma)$ cannot be contained in $B(y, d(y, \Delta_{\pi}))$ since $\pi$ is an unramified topological covering above the contractible set $B(y, d(y, \Delta_{\pi}))$ and the points $x_1$ to $x_2$ are in two different sheets.

It follows that the length of $\pi(\gamma)$ is at least  twice the distance $d(y,\Delta_\pi)$.
In turn, the inner distance $d_X(x_1, x_2)$ is greater or equal to the length of its projection by $\pi$, thus we get our inequality \eqref{eq:dist}.
\end{proof}

\begin{lemma}\label{lem1}
Let $\gamma_i\colon [0,\varepsilon)\to {\bR}^n$, $i=1,2$ be two subanalytic $\mathcal{C}^1$ distinct curves such that $\gamma_1(0)=\gamma_2(0)=0$ and $\gamma_1'(0)=\gamma_2'(0)=w$. Then there is an exponent $q\in\bQ$ such that $q>1$ and $||\gamma_1(t)-\gamma_2(t)||\leq \mathrm{const.} t^q$, for $0\leq t\ll 1$.
\end{lemma}
\begin{proof}
The function $f(t)=||\gamma_1(t)-\gamma_2(t)||$ is subanalytic, hence $f(t)=at^q+o(t^q)$ for some rational $q$ and $a>0$  (see e.g. \cite[Lemme 3]{BR}). One has $q>0$ since the curves are equal at the origin. We shall show that actually $q>1$.

It follows from \cite[Lemme 3]{BR} that there is an integer $p>0$ such that the curves $t\mapsto \gamma_1(t^p), \gamma_2(t^p)$ have analytic extensions at 0. This implies that one has Puiseux expansions  with vector coefficients:
$$
\gamma_i(t)=\sum_{\nu=1}^{+\infty} a_{i,\nu} t^{\nu/p},\quad 0\leq t\ll 1, \quad i=1,2.
$$
Since $w=\gamma_i'(0)=\lim_{t\to 0^+} \gamma_i(t)/t$ is well-defined and the same for both curves, it follows that we must actually have
$$
\gamma_i(t)=wt+\sum_{\nu=p+1}^{+\infty} a_{i,\nu} t^{\nu/p},\quad 0\leq t\ll 1,\quad i=1,2.
$$
It follows that $\gamma_1(t)-\gamma_2(t)=ct^{(p+1)/p}+o(t^{(p+1)/p})$ with some $c\in{\bR}^n$, hence $q = (p+1)/p$ satisfies the requirements.
\end{proof}

\section{Space curves}\label{curves}
\subsection{Proof of Proposition \ref{t:planecurve}}
The ``if'' part being trivial,  we focus on the ``only if'' part. 
We consider the germ of a curve $(X,0)\subset {\bC}^{n}$ and its  decomposition into irreducible components
$X=\bigcup X_j$.  
If the union of some components is not NE, then $X$ is not NE. Therefore, suppose that either $0$ is a singular point of some component $X':=X_j$, or there are two distinct nonsingular components $X_i,X_j$ tangent to each other at zero, in which case we denote $X':=X_i\cup X_j$. In both cases $X'$ is singular and $C_0(X')$ is a complex line that, by a linear change of coordinates, 
we may assume to be the first coordinate axis. 

The linear projection $\pi\colon {\bC}^n\to {\bC}$ onto this first coordinate is general with respect to $X'$,  in the sense of Definition \ref{d:generalpi}, hence $\pi_{X'}$ is a covering of degree $d\geq 2$ ramified at the origin. 

For some $v\in {\bC}$, $|v|=1$ and a small $\varepsilon>0$, let $[0,\varepsilon v]:=\{s v\mid s\in [0,\e]\}$
be a real segment. Then  the germ at 0 of the set $X'\cap \pi^{-1}([0,\varepsilon v])$ consists of $d\ge 2$ distinct real semi-analytic curve germs  which are  tangent along a unique semi-line contained in the complex line $C_0(X')$,  which projects by the general $\pi$ to the real semi-line  $\bR_{+}v$, see e.g.  \cite[\S 8.5, Propositions 1 and 2]{Ch}.  


Let us consider real Puiseux parametrizations (cf \cite[Lemma 3]{BR}) of two such real curves by $\eta_j\colon[0,\e]\ni s\mapsto (sv,\gamma_j(s))\in X'\cap\pi^{-1}([0,\e v])$, $j=1,2$, with $\gamma_j$ subanalytic of class $\mathcal{C}^{1}$, such that $\gamma_1(0)=\gamma_2(0)=0$ and $\gamma_1'(0)=\gamma_2'(0)=0$. 

By Lemma \ref{lem1}, there is an exponent $q> 1$ such that for some constant $c>0$ one has:

\[ ||\gamma_1(s)-\gamma_2(s)||\leq c s^q, \quad 0\leq s\ll 1,\] 

whereas by Lemma \ref{inner} one has:
\[d_X(\gamma_1(s),\gamma_2(s))\geq 2 d(s,0) = 2 s,  \quad 0\leq s\ll 1.\]

 Therefore:
\[
\frac{d_X(\gamma_1(s),\gamma_2(s))}{||(\gamma_1(s)-\gamma_2(s)||}\geq \mathrm{const.} s^{1-q},\quad 0<s\ll 1,
\]
where the right-hand side tends to $+\infty$ when $s\to 0^+$.   This shows that $X'$ is non-NE at zero, thus Proposition \ref{t:planecurve} is proved.
\fin

\bigskip
  
 The multiplicity of the irreducible components does not play any role in the NE problem.
  In case of plane curves, if $(X, 0)$ with reduced structure is defined by $f=0$ and $f= f_{d }+ f_{d+1}+ \ldots =0$  is the expansion into homogeneous parts, the NE problem is determined by the initial form $f_{d}$ whose zero set defines the tangent cone: 
  
 \begin{corollary}\label{c:planecurve}
 A reduced plane curve $(X, 0)\subset \bC^{2}$ is NE  if and only if the initial form $f_{d}$ of its minimal defining equation $f$ decomposes into  $d$ distinct linear terms (equivalently,  its tangent cone $C_{0}(X)$ is reduced).
 \fin
 \end{corollary}

 \begin{example}\label{e:noncusp}
 $X=\{y^2 + x^4 = 0\}$ is reduced and not irreducible. It is the union of two nonsingular irreducible components with the same tangent cone at 0.  By  Corollary \ref{c:planecurve}, $X$ is non-NE since $C_{0}(X)$ is not reduced. 
\end{example}

\section{A sectional criterion}\label{s:sectional}

Let now $(X, 0)\subset (\bC^n,0)$ be a reduced analytic germ of pure dimension $k$, with $1<k<n$. 
As introduced in \S \ref{s:prelim}, we consider a linear surjective map $\pi: \bC^{n}\to \bC^{k}$ such that it is \emph{general}  with respect to $X$ (cf Definition \ref{d:generalpi}).
The germ of the discriminant  $\Delta_\pi\subset \bC^{k}$  of $\pi_{X}$ is its ramification locus.

\begin{definition}\label{d:admissible}
A line  $\ell \in G_{1}(\bC^{k})$ is \emph{general} in the target of a linear surjection $\pi$ which is general with respect to $X$
 if  $\ell\not \subset C_0( \Delta_\pi)$.
We then say that the $(n-k+1)$-plane  $P =  \pi^{-1}(\ell)$ is \emph{admissible}.
\end{definition}

\subsection{Proof of Theorem \ref{t:sectional}}

By using Proposition \ref{t:planecurve}, we may assume without loss of generality  that there is a curve $\Gamma \subset P\cap X$ which is  singular and such that its tangent cone $C_{0}(\Gamma)\subset C_{0}(X)$ is a line.

Let us consider the restriction  $\pi_{|P} : P\to \ell$.
Then $\pi_{|P}$ is general with respect to $P\cap X$.

Like in the proof of Proposition \ref{t:planecurve},  for some unitary vector $w\in \ell\subset \bC^{k}$ one finds two distinct real curves in  $\Gamma$
with Puiseux parametrizations over the segment 
$[0,\varepsilon w]\subset \ell$, namely $\gamma_{i} : [0,\varepsilon] \to P\subset \bC^{n}$,  with $\gamma_i'(0)=w$, $i=1,2$. 
By  Lemma \ref{lem1} we get:
\[
||\gamma_1(t)-\gamma_2(t)||\leq c t^q,
\] 
for some $q>1$ and small enough $\e$, and some positive constant $c$.

 On the other hand, as  $\pi (\gamma_1(t))=\pi (\gamma_2(t)) = tw$ for any $t\in [0,\varepsilon]$, by Lemma \ref{inner} we get:
\[
d_X(\gamma_1(t),\gamma_2(t))\geq d(tw,  \Delta_\pi).
\]

Since $\ell$ is general, we have $\ell \not\subset C_0(  \Delta_\pi)$, hence we may assume that $\ell \cap \Delta_\pi = \{ 0 \}$ as germ at 0.
Therefore, there is $r>0$ such that 
 the open conical set
\[
V:=\bigcup_{s\in (0,\e)} \mathbb{B}_{sr}(s w)=\{x\in {\bC}^k\mid  ||x-sw||<sr\}
\]
is disjoint with $\Delta_\pi\cup C_0(X)$, where $\mathbb{B}_{sr}(s w)$  denotes the open ball of radius $sr$ centred 
at $sw\in \ell\subset \bC^{k}$. This ensures the inequality $d(tw,  \Delta_\pi)\geq r t$ for any $t\in [0,\varepsilon]$ and small enough $\e$. We finally get:
\[
\frac{d_X(\gamma_1(t), \gamma_2(t))}{||\gamma_1(t))-\gamma_2(t)||}\geq \frac{d(tw,  \Delta_\pi)}{c t^q} \ge \frac{r t}{c t^q} =\mathrm{const.}  t^{1-q},
\]
where the right-hand side tends to $+\infty$ when $t\to 0^+$. This shows that $X$ is non-NE at zero.
\fin

The following example shows that  the assumption ``$\pi$ is general'', i.e. that  $\ker \pi \cap C_0(X) = \{0\}$,  is essential in the above proof.

\begin{example}\label{ex:nongen}
Consider $X\colon x^2+y^2=z^3$ in ${\bC}^3$. We have $C_0(X)=\{(x,y,z)\mid x^2+y^2=0\}$. The projection $\pi(x,y,z)=(x,y)$ restricted to $X$ is finite  but not general, with ramification locus $\Delta_\pi\colon x^2+y^2=0$. For the general $\ell\colon x=0$ in ${\bC}^2$ we obtain that $X\cap P$ is the cusp $y^2=z^3$, hence a non-NE curve at 0.  
However, by applying the proof  of Theorem \ref{t:sectional} to this situation we do not arrive to an inequality that allows us to conclude. The method fails because of the lack of genericity of $\pi$.  

\end{example}

\subsection{Proof of Corollary  \ref{c:admiss}}\ \\
  Let us fix  a multiple irreducible component $C_{i}$ of the tangent cone $C_0(X)$.
  As above, we consider a general linear map  $\pi$, in the sense of Definition \ref{d:generalpi}. 
   Then the slices $X\cap P$ by admissible $(n-k+1)$-planes $P := \pi^{-1}(\ell)$ are all non-NE curves. Indeed,
the general slice $P \cap C_0(X)$ is precisely the tangent cone of the slice $P \cap X$ (see e.g. \cite[8.5, Prop. 1]{Ch}) and, by our hypothesis, this tangent cone contains the multiple complex line $P\cap C_{i}$.  We may therefore apply Theorem \ref{t:sectional} to conclude. 

\fin

\subsection{Brieskorn-type hypersurfaces}\label{brieskorn}

In \cite{BFN-normemb} the authors show that the Pham-Brieskorn singular surface 
$$
\{(x,y,z)\in{\bC}^3\mid x^a+y^b+z^b=0\},\quad 1<a<b,
$$
is non-NE at zero, provided $a$ does not divide $b$. 

This result can be generalized to a class of Pham-Brieskorn hypersurfaces, without any divisibility condition,  by using our sectional criterion: 


\begin{corollary}\label{Brieskorn}
For $n\ge 2$, let 
$$
X=\{x\in{\bC}^n\mid a_1x_1^{k_1}+\ldots+a_nx_n^{k_n}=0\}, \ \ a_i\in \bC^{*}
$$
be a Pham-Brieskorn hypersurface germ such that
 $1<k_1<k_2$ and $k_2\leq\cdots\leq k_n$. 

Then $X$ is not NE at the origin.
\end{corollary}
\begin{proof}
We may assume that $n\geq 3$, since the case of plane curves is solved by applying Corollary \ref{c:planecurve}.
The tangent cone $C_0(X)$ is $\{x_1=0\}$ with multiplicity $k_{1}>1$ since the initial form of the defining function is $a_1x_1^{k_1}$, so Corollary \ref{c:admiss} applies.
\end{proof}


\end{document}